\documentclass[11pt,reqno]{amsart}
\usepackage{enumerate,amsfonts,amsmath,amsthm,amssymb,latexsym,multicol}
\usepackage{verbatim,amscd}
\usepackage{hyperref}

\advance\textwidth by 1.2in \advance\oddsidemargin by -.6in \advance\evensidemargin by -.6in
\newtheorem*{cor}{Corollary}
\newtheorem*{lem}{Lemma}
\newtheorem*{prop}{Proposition}

\newtheorem*{thm}{Theorem}
\newtheorem{theorem}{Theorem}

\theoremstyle{definition}
\newtheorem*{defn}{Definition}
\theoremstyle{definition}

\newtheorem*{rem}{Remark}

\newenvironment{pf}{\proof}{\endproof}
\newcounter{cnt}
\newenvironment{enumerit}{\begin{list}{{\hfill\rm(\roman{cnt})\hfill}}{%
\settowidth{\labelwidth}{{\rm(iv)}}\leftmargin=\labelwidth%
\advance\leftmargin by \labelsep\rightmargin=0pt\usecounter{cnt}}}{\end{list}} \makeatletter
\def\mydggeometry{\makeatletter\dg@YGRID=1\dg@XGRID=20\unitlength=0.003pt\makeatother}
\makeatother \theoremstyle{remark}

\numberwithin{equation}{section}

\let\bwdg\bigwedge
\def\bigwedge{{\textstyle\bwdg}}

\begin{document}

\newcommand{\thmref}[1]{Theorem~\ref{#1}}
\newcommand{\secref}[1]{Section~\ref{#1}}
\newcommand{\lemref}[1]{Lemma~\ref{#1}}
\newcommand{\propref}[1]{Proposition~\ref{#1}}
\newcommand{\corref}[1]{Corollary~\ref{#1}}
\newcommand{\remref}[1]{Remark~\ref{#1}}
\newcommand{\defref}[1]{Definition~\ref{#1}}
\newcommand{\er}[1]{(\ref{#1})}

\newcommand{\wt}{\operatorname{wt}}
\newcommand{\conv}{\operatorname{conv}}
\newcommand{\cone}{\operatorname{cone}}
\newcommand{\ad}{\operatorname{ad}}
\newcommand{\Fin}{\operatorname{Fin}}
\newcommand{\argmax}{\operatorname{argmax}}
\newcommand{\supp}{\operatorname{supp}}
\newcommand{\relint}{\operatorname{relint}}

\newcommand{\zz}{\mathbb{Z}_+}
\newcommand{\V}{\mathbb{V}}
\newcommand{\vla}{\mathbb{V}^\lambda}
\newcommand{\F}{\mathbb{F}}
\newcommand{\R}{\mathbb{R}}
\newcommand{\fz}{\mathcal{F}_{\mathbb{Z}}}
\newcommand{\ff}{\mathcal{F}_{\mathbb{F}}}
\newcommand{\fr}{\mathcal{F}_{\mathbb{R}}}
\newcommand{\rz}{\mathcal{R}_{\mathbb{Z}}}
\newcommand{\rf}{\mathcal{R}_{\mathbb{F}}}
\newcommand{\liehr}{\mathfrak{h}_{\mathbb{R}}}
\newcommand{\wtvla}[1]{\wt V_{#1}(\lambda)}
\newcommand{\rhovla}[1]{\rho_{\lambda,{#1}}}

\newcommand{\comm}[1]{}

\newcommand{\tensor}{\otimes}
\newcommand{\from}{\leftarrow}
\newcommand{\disp}{\mathbf{v}}
\newcommand{\nc}{\newcommand}
\newcommand{\rnc}{\renewcommand}
\newcommand{\dist}{\operatorname{dist}}
\newcommand{\id}{\operatorname{id}}
\newcommand{\ord}{\operatorname{\emph{ord}}}
\newcommand{\sgn}{\operatorname{sgn}}
\newcommand{\gldim}{\operatorname{gl.dim}}
\newcommand{\Irr}{\operatorname{Irr}}
\newcommand{\Ht}{\operatorname{ht}}
\newcommand{\qbinom}[2]{\genfrac[]{0pt}0{#1}{#2}}
\nc{\cal}{\mathcal} \nc{\goth}{\mathfrak} \rnc{\bold}{\mathbf}
\renewcommand{\frak}{\mathfrak}
\renewcommand{\Bbb}{\mathbb}
\nc\bomega{{\mbox{\boldmath $\omega$}}} \nc\bpsi{{\mbox{\boldmath $\Psi$}}}
 \nc\balpha{{\mbox{\boldmath $\alpha$}}}
 \nc\bpi{{\mbox{\boldmath $\pi$}}}
\nc\bsigma{{\mbox{\boldmath $\sigma$}}} \nc\bcN{{\mbox{\boldmath $\cal{N}$}}} \nc\bcm{{\mbox{\boldmath $\cal{M}$}}} \nc\bLambda{{\mbox{\boldmath
$\Lambda$}}}

\newcommand{\lie}[1]{\mathfrak{#1}}
\makeatletter
\def\section{\def\@secnumfont{\mdseries}\@startsection{section}{1}%
  \z@{.7\linespacing\@plus\linespacing}{.5\linespacing}%
  {\normalfont\scshape\centering}}
\def\subsection{\def\@secnumfont{\bfseries}\@startsection{subsection}{2}%
  {\parindent}{.5\linespacing\@plus.7\linespacing}{-.5em}%
  {\normalfont\bfseries}}
\makeatother
\def\subl#1{\subsection{}\label{#1}}
 \nc{\Hom}{\operatorname{Hom}}
  \nc{\mode}{\operatorname{mod}}
\nc{\End}{\operatorname{End}} \nc{\wh}[1]{\widehat{#1}} \nc{\Ext}{\operatorname{Ext}} \nc{\ch}{\text{ch}} \nc{\ev}{\operatorname{ev}}
\nc{\Ob}{\operatorname{Ob}} \nc{\soc}{\operatorname{soc}} \nc{\rad}{\operatorname{rad}} \nc{\head}{\operatorname{head}}
\def\Im{\operatorname{Im}}
\def\gr{\operatorname{gr}}
\def\mult{\operatorname{mult}}
\def\Max{\operatorname{Max}}
\def\ann{\operatorname{Ann}}
\def\sym{\operatorname{sym}}
\def\Res{\operatorname{Res}}
\def\conv{\operatorname{conv}}
\def\und{\underline}
\def\Lietg{$A_k(\lie{g})(\bsigma,r)$}

 \nc{\Cal}{\cal} \nc{\Xp}[1]{X^+(#1)} \nc{\Xm}[1]{X^-(#1)}
\nc{\on}{\operatorname} \nc{\Z}{{\bold Z}} \nc{\J}{{\cal J}} \nc{\C}{{\bold C}} \nc{\Q}{{\bold Q}}
\renewcommand{\P}{{\cal P}}
\nc{\N}{{\Bbb N}} \nc\boa{\bold a} \nc\bob{\bold b} \nc\boc{\bold c} \nc\bod{\bold d} \nc\boe{\bold e} \nc\bof{\bold f} \nc\bog{\bold g}
\nc\boh{\bold h} \nc\boi{\bold i} \nc\boj{\bold j} \nc\bok{\bold k} \nc\bol{\bold l} \nc\bom{\bold m} \nc\bon{\bold n} \nc\boo{\bold o}
\nc\bop{\bold p} \nc\boq{\bold q} \nc\bor{\bold r} \nc\bos{\bold s} \nc\boT{\bold t} \nc\boF{\bold F} \nc\bou{\bold u} \nc\bov{\bold v}
\nc\bow{\bold w} \nc\boz{\bold z} \nc\boy{\bold y} \nc\ba{\bold A} \nc\bb{\bold B} \nc\bc{\bold C} \nc\bd{\bold D} \nc\be{\bold E} \nc\bg{\bold
G} \nc\bh{\bold H} \nc\bi{\bold I} \nc\bj{\bold J} \nc\bk{\bold K} \nc\bl{\bold L} \nc\bm{\bold M} \nc\bn{\bold N} \nc\bo{\bold O} \nc\bp{\bold
P} \nc\bq{\bold Q} \nc\br{\bold R} \nc\bs{\bold S} \nc\bt{\bold T} \nc\bu{\bold U} \nc\bv{\bold V} \nc\bw{\bold W} \nc\bz{\bold Z} \nc\bx{\bold
x} \nc\KR{\bold{KR}} \nc\rk{\bold{rk}} \nc\het{\text{ht }}

\nc\toa{\tilde a} \nc\tob{\tilde b} \nc\toc{\tilde c} \nc\tod{\tilde d} \nc\toe{\tilde e} \nc\tof{\tilde f} \nc\tog{\tilde g} \nc\toh{\tilde h}
\nc\toi{\tilde i} \nc\toj{\tilde j} \nc\tok{\tilde k} \nc\tol{\tilde l} \nc\tom{\tilde m} \nc\ton{\tilde n} \nc\too{\tilde o} \nc\toq{\tilde q}
\nc\tor{\tilde r} \nc\tos{\tilde s} \nc\toT{\tilde t} \nc\tou{\tilde u} \nc\tov{\tilde v} \nc\tow{\tilde w} \nc\toz{\tilde z}

\title[Faces of weight polytopes and a generalization of a theorem of
Vinberg]{Faces of weight polytopes and a generalization of a theorem of
Vinberg}

\author{Apoorva Khare}
\email[A.~Khare]{\tt apoorva.khare@yale.edu}
\address{Department of Mathematics, Yale University, New Haven, CT}

\author{Tim Ridenour}
\email[T.~Ridenour]{\tt tbr4@math.northwestern.edu}
\address{Department of Mathematics, Northwestern University, Evanston,
IL}

\date{\today}


\begin{abstract}

The paper is motivated by the study of graded representations of Takiff
algebras, cominuscule parabolics, and their generalizations. We study
certain special subsets of the set of weights (and of their convex hull)
of the generalized Verma modules (or GVM's) of a semisimple Lie algebra
$\lie g$. In particular, we extend a result of Vinberg and classify the
faces of the convex hull of the weights of a GVM.  When the GVM is
finite-dimensional, we answer a natural question that arises out of
Vinberg's result: when are two faces the same?

We also extend the notion of interiors and faces to an arbitrary subfield
$\F$ of the real numbers, and introduce the idea of a weak $\F$--face of
any subset of Euclidean space. We classify the weak $\F$--faces of all
lattice polytopes, as well as of the set of lattice points in them. We
show that a weak $\F$--face of the weights of a finite-dimensional $\lie
g$--module is precisely the set of weights lying on a face of the convex
hull.
\end{abstract}
\maketitle

\section{Introduction}

In this note, we study the faces of the convex hull of the weights of a
highest weight representation $V$  of a complex semisimple Lie algebra
$\lie g$. The classification of the faces in the case when $V$ is a
simple finite-dimensional representation of $\lie g$ had been obtained by
Vinberg \cite{Vin}. Roughly speaking, his result states that a face of
the weight polytope of a simple finite-dimensional representation is
determined by a pair consisting of an element of the Weyl group and a
subset of the set of simple roots.
Our results extend (and recover) those of Vinberg's for arbitrary
generalized Verma modules. Our methods, however, are completely different
and rely on algebra and convexity theory. In particular, we are able to
work with convex linear combinations of the weights, where the
coefficients are in an arbitrary subfield of the real numbers.
We are also able to answer a natural question arising from Vinberg's
result: namely, when do two different pairs give rise to the same face of
the weight polytope of a finite-dimensional simple Lie algebra.

This paper was motivated by the results in \cite{CG2} (which are further
extended in \cite{CKR}) on representations of Takiff algebras and their
generalizations. In those papers, one showed that one could associate
Koszul algebras in a natural fashion, to certain subsets of the set of
weights of a finite-dimensional representation of a semisimple Lie
algebra. In this paper, we show that the conditions on these subsets is
exactly equivalent to requiring the subset to be the maximal subset of
weights contained in a face. This description generalizes and makes
uniform the results of \cite{CDR}, where the case of the adjoint
representation was analyzed.

\subsection*{Organization}

The paper is organized as follows. In \secref{S2}, we study generalized
Verma modules. These are a family of highest weight $\lie g$-modules,
that run from all Verma modules at one end, to all finite-dimensional
simple modules at the other. The convex hull of their set of weights
turns out always to be a polyhedron, and our main goal in this section is
to classify their faces, in terms of describing the vertices and the
extremal rays. This generalizes Vinberg's result from \cite{Vin}.

For the rest of the paper, we focus on finite-dimensional $\lie
g$-modules $V$. We wish to study the subsets of weights of $V$, which lie
on faces of the convex hull of all weights. To that end, we introduce the
notion of a {\it weak face}, over an arbitrary subfield $\F \subset \R$.
Among these weak $\F$-faces, we then consider {\it positive} weak
$\F$-faces.
In \secref{S3}, we classify the (positive) weak $\F$-faces of $V$. This
generalizes results from \cite{CDR,CG2}, which addressed the example of
$V = \lie g$.

In \secref{S4}, we study (positive) weak $\F$-faces of arbitrary subsets
$X \subset \R^n$. Our main results here concern the case when the convex
hull of $X$ is a polyhedron. In this case, the (positive) weak $\F$-faces
are precisely the elements of $X$ that lie on a proper face of the
polyhedron - in other words, that maximize a linear functional, with
finite (positive) maximum.

Finally, in \secref{S5}, we prove our results from \secref{S3}, using the
techniques developed in \secref{S4}.

\subsection*{Acknowledgements}
The authors are especially grateful to Vyjayanthi Chari for extremely
valuable discussions and her many inputs and suggestions, that helped
bring this manuscript to its present form. The first author would also
like to thank Michel Brion, Shrawan Kumar, and Olivier Mathieu for
valuable discussions.

\section{Results on generalized Verma modules}\label{S2}

Throughout this paper, we let $\Bbb{R}$ (respectively $\Bbb{Q}$,
$\Bbb{Z}$) denote the real numbers (respectively the rationals, and the
integers). For any subset $R \subset \Bbb{R}$, we let $R_+ := R \cap [0,
\infty),\ R_{>0} := R \cap (0, \infty)$. If $A,B \subset V$ are subsets
of an abelian group $(V,+)$, we define their Minkowski sum to be $A+B :=
\{ a+b : a \in A, b \in B \} \subset V$. (If $A = \{ a \}$, we may also
write this as $a+B$.) Similarly, $-B := \{ -b : b \in B \}$.

\subsection{}

Fix a complex semisimple Lie algebra $\lie{g}$ of rank $n$ and a Cartan
subalgebra $\lie h$ of $\lie g$, and let $\Phi \subset \lie h^*$ be the
set of roots of $\lie g$ with respect to $\lie h$.
Set $I = \{ 1, \cdots, n \}$ and fix a set $\{ \alpha_i : i \in I \}$ of
simple roots.
Denote by $\Phi^+$ the corresponding set of positive roots.
Let $\kappa $ be the Killing form on $\lie{g}$; recall that its
restriction to $\lie h$ induces a positive definite inner product $(\ ,\
)$ on the real span $\liehr^*$ of $\Phi^+$.
Let $\{ \omega_i : i \in I \}$ be the basis of $\lie h^*$ which satisfies
$2 (\alpha_i, \omega_j) = \delta_{i,j}(\alpha_i,\alpha_i)$. Since the
Killing form is nondegenerate, it induces an identification of $\liehr$
with $\liehr^*$. Define $h_{\alpha_i} \in \liehr$ to be the vector
identified with $2 \alpha_i / (\alpha_i, \alpha_i)$; these vectors form
an $\R$-basis of $\liehr$.

The root lattice $Q$ (respectively, weight lattice $P$) is the integer
span of the simple roots $\alpha_i$ (respectively, fundamental weights
$\omega_i$), while $Q^+$ (respectively, $P^+$) is the $\Bbb Z_+$-span of
the simple roots (respectively, fundamental weights).
Given a subset $J$ of $I$, let $Q_J$ (respectively $P_J$) be the $\Bbb
Z$-span of the simple roots $\{ \alpha_j : j \in J \}$ (respectively,
the fundamental weights $\{ \omega_j : j \in J \}$), and set $\Phi_J^+ :=
\Phi^+ \cap Q_J,\ P_J^+ := P^+ \cap P_J,\ Q_J^+ := Q^+ \cap Q_J$.

Given any $\lambda \in \lie h^*$, say $\lambda = \sum_{i \in I} r_i
\omega_i$ with all $r_i \in \Bbb R$, we set
$$\supp \lambda := \{ i \in I : r_i \ne 0 \}, \quad J_\lambda := \{ i \in
I: \lambda(h_{\alpha_i}) \in \zz \}.$$

\noindent Clearly, $\lambda \in P^+$ if and only if $J_\lambda = I$.
Finally, let  $W$ be the Weyl group of $\Phi$, namely the subgroup of
$\rm{Aut}(\liehr^*)$ generated by the simple reflections $\{ s_i : i \in
I \}$. Note that the inner product  $(\ ,\ )$ on $\liehr^*$ is
$W$-invariant.

\subsection{}

Fix a Chevalley basis $\{ x_\alpha^\pm, h_i = h_{\alpha_i} : \alpha\in
\Phi^+, 1 \leq i \leq n \}$ of $\lie g$, set $\lie n^\pm =
\bigoplus_{\alpha \in \Phi^+} \Bbb C x_\alpha^\pm$, and write
$$\lie g=\lie n^- \oplus \lie h \oplus \lie n^+.$$

\noindent The subalgebras $\lie n^\pm_J$ are defined in the obvious way.
Let $\lie{p}_J$ be the parabolic Lie subalgebra of $\lie{g}$, defined as
follows:
$$\lie p_J = \lie n^-_J \oplus \lie h \oplus \lie n^+, \quad \lie m_J =
\lie n^-_J \oplus \lie h\oplus\lie n^+_J, \quad \lie u^\pm_J =
\bigoplus_{\alpha \in \Phi^+ \setminus \Phi^+_J} \Bbb C x^\pm_\alpha,$$

\noindent where $\lie m_J$ is reductive with semisimple part $\lie g_J$,
and $\lie u^\pm_J$ is nilpotent.
The subgroup $W_J$ of $W$ generated by $\{ s_j : j \in J \}$ is the Weyl
group of $\lie g_J$, and we set $\rho_J = \sum_{j \in J} \omega_j$.
The following is standard, but we isolate it in the form of a Lemma,
since it is used frequently in the paper.

\begin{lem}\label{L22}
For $w \in W_J$ and $i \notin J$, we have $w \alpha_i \in \Phi^+$, and
hence $w \alpha \in \Phi^+$ for all $\alpha \in \Phi^+ \setminus
\Phi^+_J$.
\end{lem}

Given any Lie algebra $\lie a$, we let $\bu(\lie a)$ be the universal
enveloping algebra of $\lie a$. The Poincare--Birkhoff--Witt theorem
gives us an isomorphism of vector spaces:
$$\bu(\lie g) \cong \bu(\lie n^-) \otimes \bu(\lie h) \otimes \bu(\lie
n^+).$$

\subsection{}

We now recall the definition and elementary properties of the {\it
generalized Verma modules}. Recall that a weight module $V$ for a
reductive Lie algebra $\lie a $ with Cartan subalgebra $\lie t$ is one
which has a decomposition
$$V = \bigoplus_{\mu \in \lie t^*} V_\mu,$$

\noindent where $V_\mu = \{v \in V : hv = \mu(h)v,\ \ \forall\ \ h \in
\lie t\}$. We set $\wt V = \{\mu \in \lie t^*: V_\mu \ne 0\}$.

Given $\lambda \in \lie h^*$ and $J\subset J_\lambda$, the generalized
Verma module $M(\lambda,J)$ is the $\lie g$-module generated by an
element $m_\lambda$ with defining relations:
$$\lie n^+ m_\lambda = 0, \quad h m_\lambda = \lambda(h) m_\lambda, \quad
(x^-_\alpha)^{\lambda(h_\alpha)+1} m_\lambda=0,$$

\noindent for all $h \in \lie h$ and $\alpha \in \Phi^+_J$. The following
is standard - see \cite{Kumar}:

\begin{prop}\label{elemgen}
Let $\lambda \in \lie h^*$ and $J \subset J_\lambda$.
\begin{enumerit}
\item The $\lie g$-module $M(\lambda,J)$ is a free $\bu(\lie
u_J^-)$-module, and $\dim \bu(\lie m_J) m_\lambda < \infty$. In
particular, $\wt(\bu(\lie m_J) m_\lambda)$ is a finite subset of $\lie
h^*$, and
$$\wt M(\lambda, J) = \wt (\bu(\lie m_J) m_\lambda) - \left\{
\sum_{\alpha \notin \Phi^+_J} r_\alpha \alpha : r_\alpha \in \Bbb Z_+
\right\}.$$

\item The set $\wt M(\lambda, J)$ is $W_J$-invariant.
\end{enumerit}
\end{prop}

In the special case when $\lambda\in P^+$, the module $M(\lambda, I)$ is
the irreducible finite-dimensional module with highest weight $\lambda$.

\subsection{}\label{S24}

Given any subset $X$ of $\liehr^*$ we let $\conv_{\Bbb R}(X)$ be the
convex hull of $X$; i.e.,
$$\conv_{\Bbb R}(X)=\left\{\sum_{s=1}^k r_s x_s\ :\ k \in \Bbb Z_+, r_s
\in \Bbb R_+, x_s \in X,\ \sum_{s=1}^k r_s = 1 \right\}.$$

\noindent Also define $\cone_{\Bbb R}(X)$ to be the cone of $X$, i.e.,
$$\cone_{\Bbb R}(X) = \left\{ \sum_{s=1}^k r_s x_s \ : \ k \in \Bbb Z_+,
r_s \in \Bbb R_+, x_s \in X \right\}.$$

\begin{prop}\label{P24}
For $\lambda \in \liehr^*$ and $J \subset J_\lambda$, we have
$$\conv_{\Bbb R}(\wt M(\lambda, J)) = \conv_{\Bbb R}(\wt \bu(\lie
m_J)m_\lambda) - \cone_{\Bbb R}(\Phi^+ \setminus \Phi_J^+),$$

\noindent and hence $\conv_{\Bbb R}(\wt M(\lambda, J))$ is a
$W_J$-invariant subset of $\liehr^*$.
\end{prop}

\begin{pf}
It is clear (by \propref{elemgen}) that $\conv_{\Bbb R}(\wt M(\lambda,
J))$ is contained in the right hand side. For the reverse inclusion, let
$$\mu = \sum_k r_k \mu_k - \sum_{\alpha \in \Phi^+ \setminus \Phi_J^+}
m_\alpha \alpha, \quad \mu_k \in \wt \bu(\lie m_J)m_\lambda, \quad r_k,
m_\alpha \in \Bbb R_+, \quad \sum_k r_k = 1.$$

\noindent If $m_{\alpha} = 0$ for all $\alpha$, then we are done since
$\mu_k\in\wt(M(\lambda, J))$ for all $k$. Hence, we may assume that
$m_{\alpha} >0$ for some $\alpha \in \Phi^+\setminus \Phi_J^+$.
Furthermore, we may assume without loss of generality that $r_1 \neq 0$.
Thus, we can write
\[\mu = \sum_{k \neq 1} r_k \mu_k + r_1(\mu_1 - \sum_\alpha
\frac{m_{\alpha}}{r_1} \alpha). \]

Choose $t \in \Bbb Z_+$ such that $\displaystyle r = \sum_\alpha
\frac{m_{\alpha}}{r_1} \le t$.  Since $\mu_1 - t \alpha \in \wt
M(\lambda, J)$, the claim follows by noting that
$$\mu = \sum_{k \neq 1} r_k \mu_k + \left( r_1 - \frac{r_1 r}{t} \right)
\mu_1 + \sum_\alpha \frac{m_{\alpha}}{t} (\mu_1 - t \alpha) \in
\conv_{\Bbb R}(\wt M(\lambda, J)).$$

\noindent The fact that $\conv_{\Bbb R}(\wt M(\lambda, J))$ is
$W_J$-invariant is immediate from Proposition \ref{elemgen}.
\end{pf}

\subsection{}

We now need some more notions from convexity theory. Given $v,w \in
\liehr^* = \R^n$, define the (affine) hyperplane and the corresponding
half-space as follows:
$$H(v,w) := \{u \in \R^n \ | \ v\cdot(u-w) = 0\}, \ \ H^+(v,w) := \{u \in
\R^n \ | \ v\cdot(u-w) \geq 0\}.$$

\noindent Let $\P$ be a (nonempty) subset of $\R^n$. We say that $H(v,w)$
is a {\it supporting hyperplane} of $\P$ if
$$\P\subset H^+(v,w) \ {\rm and} \ \P\cap H(v,w) \neq \emptyset.$$

\noindent A {\it face} of $\P$ is $\P$ or the intersection of $\P$ with a
supporting hyperplane. We will say that $\P$ is a {\it polyhedron} if it
is the intersection of a finite number of affine half-spaces, and a
bounded polyhedron is a {\it polytope}. The following is standard; see
\cite{Zie}, for instance:

\begin{thm}[Decomposition Theorem]\label{T25}
Let $\P$ be a subset of $\R^n$. Then,
\begin{enumerit}
\item {\rm (Weyl-Minkowski Theorem.)} $\P$ is a polytope if and only if
$\P = \conv_{\R}(U)$ for some finite subset $U \subset \R^n$.

\item {\rm (Finite Basis Theorem.)} $\P$ is a polyhedron if and only if
$\P = \conv_{\R}(U)+ \cone_{\R}(V)$ for some finite sets $U,V \subset
\R^n$.
\end{enumerit}
\end{thm}

In particular, the convex hull of the union of a finite set with a polytope
is also a polytope.

Using the Decomposition Theorem, we have the following corollary to
\propref{P24}:

\begin{cor}
The set $\conv_{\Bbb R}(\wt \bu(\lie m_J) m_\lambda)$ is a convex
polytope, and $\conv_{\Bbb R}(\wt M(\lambda, J))$ is a convex polyhedron
in $\liehr^*$.
\end{cor}

\subsection{}

One of the main results of this paper is the following:

\begin{theorem}\label{Tvin}
Let $\lambda \in \liehr^*$, $J\subset J_\lambda$, and let
$F \subset \conv_\R(\wt M(\lambda, J))$.
Then $F$ is a face of $\conv_\R(\wt M(\lambda, J))$ if and only if there
exists a subset $I_0$ of $I$ and $w \in W_J$, such that
$$wF = \conv_\R(\wt M(\lambda,J) \cap (\lambda - Q_{I_0}^+)).$$
\end{theorem}

\begin{pf}
Let $F$ be a face of $\conv_{\R}(\wt M(\lambda,J))$. By \lemref{L2}, $F$
maximizes some linear functional $\varphi \in (\R^n)^*$ in
$\conv_{\R}(\wt M(\lambda,J))$. Let $\nu \in \lie{h}_{\R}^*$ be such that
$\varphi(\mu) = (\nu, \mu)$ for all $\mu \in \lie{h}_{\R}^*$. Choose $w
\in W_J$ such that $(w(\nu), \alpha_j) \geq 0$ for all $j \in J$. Notice
that $wF$ maximizes the inner product $(w(\nu), -)$ and, hence, is a face
of $\conv_{\R}(\wt M(\lambda, J))$ . 

Suppose that $(w(\nu), \alpha_i) < 0$ for some $i \in I \setminus J$.
Since $i \not\in J$, $\lambda - r\alpha_i \in \wt M(\lambda,J)$ for all
$r \in \Bbb Z_+$. However, $(w(\nu), \lambda - r\alpha_i) = (w(\nu),
\lambda) - r(w(\nu), \alpha_i)$ can be arbitrarily large, which
contradicts $(w(\nu), -)$ having a finite maximum in $\conv_{\R}(\wt
M(\lambda,J))$. Thus, $w(\nu)$ must be in the fundamental Weyl chamber. 

Write $w(\nu) = \sum_{i \in I} a_i \omega_i$ with $a_i \geq 0$ for all $i
\in I$. Letting $I_0 = \{ i \in I \ | \ a_i = 0\}$, it is clear that $wF
= \conv_{\R}(\wt M(\lambda, J) \cap (\lambda - Q_{I_0}^+))$.

For the converse, let $\rho_{I \setminus I_0} = \sum_{i \in I \setminus
I_0} \omega_i$, and consider the linear functional $\varphi$ given by
$$\varphi(\mu) = (\rho_{I \setminus I_0}, \mu).$$

The subset of $\conv_{\R}(\wt M(\lambda, J))$ that maximizes $\varphi$ is
precisely $\conv_{\R}(\wt M(\lambda,J) \cap (\lambda - Q_{I_0}^+))$.
Hence $\conv_{\R}(\wt M(\lambda,J) \cap (\lambda - Q_{I_0}^+))$ is a face
of $\conv_{\R}(\wt M(\lambda,J))$. 

Suppose that $F \subset \conv_{\R}(\wt M(\lambda, J))$ and $wF =
\conv_{\R}(\wt M(\lambda,J)\cap(\lambda-Q_{I_0}^+))$ for some $w \in
W_J$. Notice that $F$ maximizes the linear functional $\varphi\circ w$
where $\varphi(\mu) = (\rho_{I\setminus I_0}, \mu)$ as above; therefore,
$F$ is a face of $\conv_{\R}(\wt M(\lambda, J))$ by \lemref{L2}.
\end{pf}

As a consequence, we obtain information about the set of weights of
$M(\lambda, J)$ that lie in a face. 

\begin{cor}
Let $\lambda \in \liehr^*$, $J\subset J_\lambda$, and suppose that $F$ is
a face of $\conv_\R (\wt M(\lambda, J))$. There exist $w \in W_J$ and
$I_0 \subset I$ such that
$$w(F \cap \wt M(\lambda, J))= \wt M(\lambda, J) \cap (\lambda -
Q_{I_0}^+).$$
\end{cor}

\begin{proof}
By \thmref{Tvin}, there exist $w \in W_J$ and $I_0 \subset I$, such that
\[ wF = \conv_\R(\wt M(\lambda,J) \cap (\lambda - Q^+_{I_0})). \]

\noindent By Proposition \ref{elemgen}, $\wt M(\lambda,J)$ is
$W_J$-invariant, and so we have
\[ wF \cap \wt M(\lambda,J) = w(F \cap \wt M(\lambda,J)), \]

\noindent and the corollary follows if we prove that
\[ \conv_\R(\wt M(\lambda, J) \cap (\lambda - Q^+_{I_0})) \cap \wt
M(\lambda, J) = ( \lambda - Q^+_{I_0}) \cap \wt M(\lambda, J). \]

It is clear that the right hand side is contained in the left hand side.
For the reverse inclusion, given $\mu \in \conv_\R(\wt M(\lambda, J) \cap
(\lambda - Q^+_{I_0})) \cap \wt M(\lambda, J)$, we write:
$$\mu = \lambda - \sum_{i \in I} n_i \alpha_i = \sum_{s=1}^k a_s (\lambda
- \sum_{i \in I_0} m_{si} \alpha_i),$$

\noindent where $n_i \in \Bbb Z_+$ for $i \in I$, $m_{si} \in \Bbb Z_+$
for $1\le s\le k$ and $i\in I_0$, and $a_s \in \R_+$, with $\sum_{s=1}^k
a_s = 1$. Using the linear independence of $\alpha_i$, $i\in I$, we see
immediately that $n_i = 0\ \forall i \notin I_0$, and hence $\mu\in
\lambda - Q^+_{I_0}$. The reverse inclusion is proved, and we are done.
\end{proof}

\begin{rem}
In the case when $\lambda\in P^+$ and $J = J_\lambda = I$, the Theorem is
proved in \cite{Vin}. Our proof as we mentioned in the introduction is
quite different.
\end{rem}

\subsection{}\label{S27}

Another corollary of the above theorem is:

\begin{cor}
$F$ is a face of $\conv_\R(\wt(M(\lambda,J)))$ if and only if
$$F = \conv_{\R}(w(\wt \bu(\lie m_{I_0\cap J}) m_\lambda)) - \cone_{\Bbb
R} w(\Phi_{I_0}^+ \setminus \Phi_{I_0 \cap J}^+)$$

\noindent for some $w \in W_J$ and $I_0 \subset I$.
\end{cor}

\begin{pf}
We first prove the following statement (which generalizes Proposition
\ref{P24}):

For all $I_0, J \subset I$,
\begin{equation}\label{E27}
\conv_\R(\wt M(\lambda,J) \cap (\lambda-Q_{I_0}^+)) = \conv_{\R}(\wt
\bu(\lie m_{I_0 \cap J}) m_\lambda) - \cone_{\Bbb R} (\Phi_{I_0}^+
\setminus \Phi_{I_0 \cap J}^+).
\end{equation}

To prove this equation, note that $\wt(M(\lambda, J)) \cap (\lambda -
Q_{I_0}^+)$ is the same as the set of weights of the
$\lie{g}_{I_0}$-submodule $\bu(\lie{g}_{I_0}) m_{\lambda}$.
Restricting our attention to $\lie{g}_{I_0}$, we see that, as in
\propref{elemgen},
$$\conv_{\Bbb R}(\wt \bu(\lie{g}_{I_0}) m_{\lambda}) = \conv_{\R}(\wt
\bu(\lie{m}_{I_0 \cap J}) m_\lambda) - \left\{ \sum_{\alpha \in
\Phi_{I_0}^+ \setminus \Phi_{I_0 \cap J}^+} r_\alpha \alpha\ : \ r_\alpha
\in \R_+ \right\}.$$ 

This proves Equation \eqref{E27}. Now, by Theorem \ref{Tvin}, if $F$ is a
face, there exist $w \in W_J$ and $I_0 \subset I$ such that $w^{-1}(F) =
\conv_\R(\wt(M(\lambda,J)) \cap (\lambda-Q_{I_0}^+))$. The result then
follows from Equation \eqref{E27} and the linearity of $w$.
\end{pf}

We claim that the above corollary is a special case of the following more
general result - which also generalizes a result of Vinberg in
\cite{Vin}.

\begin{prop}
Let $\lambda \in \liehr^*$, and let $J \subset I$. Suppose
$\bp(\lambda,J)$ is the polyhedron in $\lie{h}_{\R}^*$ given by
$$\bp(\lambda,J) = \conv_{\R}(W_J(\lambda)) - \cone_{\Bbb R}
(\Phi^+\setminus\Phi_J^+).$$

\noindent Then $F$ is a face of $\bp(\lambda,J)$ if and only if
$$w(F) = \conv_{\R}(W_{J\cap I_0}(\lambda')) - \cone_{\Bbb
R}(\Phi_{I_0}^+ \setminus \Phi_{J \cap I_0}^+),$$

\noindent for some $w \in W_J$ and $I_0 \subset I$, where $\lambda' \in
W_J(\lambda)$ satisfies $\lambda'(h_j) \geq 0 \ \forall \ j \in J$.
\end{prop}

\begin{pf}
The proof goes through as in the proof of \thmref{Tvin} once we note that
if $\mu \in \bp(\lambda,J)$, then $\lambda' - \mu \in \R_+\Delta$. For
this, it suffices to show that $\lambda' - w(\lambda) \in \R_+\Delta$ for
all $w\in W_J$. 

Consider the partial order on $\lie{h}_{\R}^*$ given by $\mu \preccurlyeq
\nu$ if and only if $\nu - \mu \in \R_+\Delta$. Recall that the
intersection of the fundamental Weyl chamber and the Weyl orbit of any
nonzero element in $\lie{h}_{\R}^*$ contains exactly one element. In
particular, if $w \in W_J$ and $w(\lambda) \neq \lambda'$, then
$w(\lambda)(h_j) < 0$ for some $j \in J$. Then, $w(\lambda) \prec
s_j(w(\lambda))$. 

Since the set $W_J(\lambda)$ is finite, it must contain a maximal element
with respect to the partial order. We have shown that $w(\lambda) \neq
\lambda'$ is not maximal, so $\lambda'$ must be the unique maximal
element in $W_J(\lambda)$ in this partial order. In other words,
$\lambda'-w(\lambda) \in \R_+\Delta$ for all $w\in W_J$.
\end{pf}

\section{Results on finite-dimensional modules}\label{S3}

Our other main results involve extending the notion of convexity and
faces to arbitrary subfields $\F$ of $\R$. We first note that for any $X
\subset \R^n$ and subfield $\F \subset \R$, we can define the $\F$-convex
hull, $\conv_\F(X)$, and $\F$-cone, $\cone_\F(X)$, similar to the case
when $\F = \R$ in \secref{S24}. Next, we extend the notion of relative
interior as follows: the {\it $\F$-relative interior} of $Y =
\conv_{\F}(X)$ is the subset
$$\relint_\F(\conv_{\F}(X)) = \{x \in Y\ |\ \forall y \in Y,\ \exists z
\in Y,\ t \in \F\cap (0,1) \mbox{ such that } x = ty + (1-t)z\}.$$

It is clear that the $\R$-relative interior of a polyhedron does not
intersect any proper face of the polyhedron.

\begin{rem}
For the remainder of the paper, we will freely use
$\relint(\conv_{\F}(X))$ to indicate the $\F$-relative interior. Strictly
speaking, this is an abuse of notation: for example, if $X$ is
$\R$-convex in $\R^n$, then $\relint_\F(\conv_\F(X)) = \relint_\F(X)$
depends on $\F$. However, we only work with $\relint_\F(\conv_\F(X))$ in
this paper.
\end{rem}

We now come to the two main new concepts in this paper. We are interested
in studying certain subsets of sets $X$, that are related to the faces of
$\conv_\R(X)$. Among these, we further distinguish some of them.

\begin{defn}
Fix a subset $X \subset \R^n$, and a subfield $\F \subset \R$.
\begin{enumerit}
\item We say that $Y \subset X$ is a {\it weak $\F$-face} of $X$ if for
every $U \subset X$,
$$\conv_{\F}(Y) \cap \relint_\F(\conv_{\F}(U)) \neq \emptyset \implies U
\subset Y.$$

\item A weak $\F$-face $Y \subset X$ is a {\it positive weak $\F$-face}
of $X$ if for every $U \subset X$,
$$\conv_{\F}(Y) \cap \relint_\F(\conv_{\F}(U\cup\{0\})) = \emptyset.$$ 
\end{enumerit}
\end{defn}

\noindent (Clearly, $X$ is always a weak $\F$-face of $X$.) As we will
see below, (positive) weak $\F$-faces of $X$ are closely related to the
faces of $\conv_\R(X)$, when the latter is a polyhedron. Our main results
now characterize the (positive) weak $\F$-faces of (the set of weights
of) finite-dimensional $\lie g$-modules.

\begin{theorem}\label{T2}
Suppose $V$ is a finite-dimensional $\lie g$-module, and $\F$ is a
subfield of $\R$. Then either $\wt V = \{ 0 \}$, or the following are
equivalent for a proper subset $Y \subset \wt V$:
\begin{enumerit}
\item $Y$ is a positive weak $\F$-face of $\wt V$.

\item $Y$ is a weak $\F$-face of $\wt V$.

\item $Y = F \cap \wt V$, for some proper face $F$ of $\conv_\R(\wt V)$.
\end{enumerit}
\end{theorem}

As we see in \lemref{L2}, faces of the polytope $\conv_\R(\wt V)$ are
precisely maximizers of linear functionals; thus, our result generalizes
a result in \cite{CDR,CG2}, which was stated only for the simple module
$V = \lie g$, and proved using a case-by-case analysis involving long and
short roots.

Our next result is a characterization of precisely which subsets of $\wt
V(\lambda)$ form (positive) weak $\F$-faces, and once again, it
generalizes (and recovers) the example of $V(\lambda) = \lie g$ that was
studied in \cite{CDR}. Moreover, it combines features from both the
theorems above (Theorems \ref{Tvin} and \ref{T2}).

\begin{theorem}\label{T32}
Suppose $\F \subset \R$, $0 \neq \lambda \in P^+$, and $V(\lambda) =
M(\lambda, I)$ is simple. Then the following are equivalent for a subset
$Y \subset \wt V(\lambda)$:
\begin{enumerit}
\item There exist $w \in W$ and $I_0 \subset I$ such that $wY = \wtvla{}
\cap (\lambda - Q^+_{I_0})$.

\item $Y$ is a weak $\F$-face of $\wt V(\lambda)$.

\item Let $\rho_Y := \sum_{y \in Y} y$. Then $Y$ is the
maximizer in $\wt V(\lambda)$ of the linear functional $(\rho_Y, -)$.
\end{enumerit}

\noindent If, furthermore, $Y \neq \wt V(\lambda)$, then $\rho_Y \in
P^+$ and the functional $(\rho_Y, -)$ has positive maximum on $\wt
V(\lambda)$.
\end{theorem}

\noindent Note that both of these results (Theorems \ref{T2} and
\ref{T32}) are independent of $\F$. Moreover, the vector $\rho_Y$ has a
geometric interpretation: it is a positive rational multiple of the
``center of mass" of the face $\conv_\R(Y)$ of $\conv_\R(\wt
V(\lambda))$.

To state our last result, we need two more pieces of notation.

\begin{defn}
Given $\lambda \in \lie h^*$ and $I_0 \subset I$, define $\wtvla{I_0} :=
\wt V(\lambda) \cap (\lambda - Q^+_{I_0})$, and $\rhovla{I_0} :=
\rho_{\wtvla{I_0}}$.
\end{defn}

We now answer a natural question arising from Vinberg's result.

\begin{theorem}\label{T33}
If $0 \neq \lambda \in P^+$, then for any $I_1, I_2 \subset I$,
$\wtvla{I_1} = \wtvla{I_2}$ if and only if $\rhovla{I_1} =
\rhovla{I_2}$, if and only if the sets of vertices coincide:
$W_{I_1}(\lambda) = W_{I_2}(\lambda)$.
\end{theorem}

\section{Faces of polyhedra}\label{S4}

Our main goal in this section is to develop the techniques that will be
needed to prove the theorems in \secref{S3}. In particular, we will show
the following result.

\begin{thm}
Suppose $\conv_\R(X)$ is a polyhedron for $X \subset \F^n$. Then $Y
\subset X$ is a weak $\F$-face if and only if $Y$ maximizes a linear
functional in $X$. If instead, $\conv_R(X \cup \{ 0 \})$ is a polyhedron for
$X \subset \F^n$, then $Y \subset X$ is a positive weak $\F$-face
if and only if $Y$ maximizes a linear functional in $X$, and this maximum
value is positive.
\end{thm}

This result also explains the choice of terminology behind (positive)
weak $\F$-faces.

\subsection{}


The following lemma will be crucial in our examination of these sets.

\begin{lem}\label{L1}
Suppose $X \subset \F^n$. If $u$ is in the $\F$-relative interior of
$\conv_{\F}(X)$ and $x_0 \in X$, then there exist $m>0$, $r_0, r_1,..r_m
\in \F_{>0}$, and $x_1, ..., x_m \in X$ such that
$$u = r_0x_0 + \sum_{j=1}^m r_jx_j, \ \ r_0 + \sum_{j=1}^m r_j = 1.$$
\end{lem}

\begin{pf}
Since $u$ is in the interior of $\conv_{\F}(X)$, we can find $t \in
\F_{>0}$ such that $u = t x_0 + (1-t) x'$, where $x' \in \conv_\F(X)$.
By definition of $\conv_{\F}(X)$, we can write $x' = \sum_{j=1}^m s_jx_j$
for some $x_j \in X$ and $s_j \in \F_{>0}$ such that $\sum_{j=1}^m s_j =
1$. Solving for $u$, we have
$$u = t x_0 + \sum_{j=1}^m s_j(1-t) x_j.$$

\noindent Setting $r_0 = t$ and $r_j = s_j(1-t)$ gives the result.
\end{pf}

We remark that if $X$ is not a singleton, we may choose all $x_0, x_1,
\dots, x_m$ to be distinct from $u$: we start by choosing any $x_0 \neq
u$ in $\conv_\F(X)$, and proceed as above. Now if $x_j = u$ for some
$j>0$, then we simply subtract $r_j u$ from both sides, and divide by
$1-r_j$.

\subsection{}

The following lemma will also be used frequently.

\begin{lem}\label{L2}
Let $\P \subset \R^n$ be nonempty. A nonempty subset $F \subset \P$ is a
face of $\P$ if and only if $F$ is the subset of $\P$ that maximizes some
linear functional $\varphi \in (\R^n)^*$.
\end{lem}

\begin{pf}
If $F$ is a face of $\P$, $F = \P \cap H(v,w)$ for some supporting
hyperplane $H(v,w)$. Define $\varphi: \R^n \to \R$ by $\varphi(u) =
-v\cdot u$. It is easy to see that $\varphi$ is maximized in $\P$
precisely on $F$.

Similarly, if $\varphi \in (\R^n)^*$ is maximized in $\P$ on $F$, choose
$v$ such that $\varphi(u) = -v \cdot u$ for all $u \in \R^n$. Let $w \in
F$. Then, $F = \P \cap H(v,w)$.
\end{pf}

\begin{prop}\label{P2}
Suppose that $X \subset \F^n$. Then $\conv_\F(X) = \conv_\R(X) \cap
\F^n.$
\end{prop}

\begin{pf}
The inclusion $\conv_\F(X) \subset \conv_\R(X)\cap \F^n$ is obvious. 

Suppose that $u \in \conv_\R(X) \cap \F^n$. Then, $u \in \conv_\R(U)$ for
some finite subset $U \subset X$. By Caratheodory's theorem, $u$ is in
some $r$-simplex $S \subset \conv_\R(U)$, such that the vertices of $S$
are a subset of $U$. 

Let $\{ s_0, s_1, \dots, s_r \}$ be the vertices of $S$. Then, $u =
\sum_{i=0}^r n_i s_i$ for some $n_i \in \R_+$ with $\sum_{i=0}^r n_i =
1$. 

Let $\psi$ be an $\F$-affine transformation of $\R^n$ such that
$\psi(s_0) = {\mathbf 0}$ and $\psi(s_i) = \mathbf{e}_i,\ 1 \leq i \leq
r$, where $\{ \mathbf{e}_i : 1 \leq i \leq n \}$ are the standard basis
vectors in $\R^n$. It is easy to see that
$$\psi(u) = \sum_{i=1}^r n_i\mathbf{e}_i.$$

\noindent Since $\psi$ is $\F$-affine, $\psi(u) \in \F_+^n$, and $n_i \in
\F_+$ for $i>0$. Furthermore, $n_0 = 1 - \sum_{i=1}^r n_i \in \F_+$, so
$u \in \conv_{\F}(X)$.
\end{pf}

\begin{cor}
Suppose $X \subset \F^n$, and $F$ is a face of $\conv_\R(X)$. Then $F
\cap \F^n$ is a face of $\conv_\F(X)$.
\end{cor}

\begin{pf}
Let $H(v,w)$ be a supporting hyperplane for $\conv_{\R}(X)$ such that $F
= \conv_{\R}(X) \cap H(v,w)$. Then,
$$F \cap \F^n = \conv_{\R}(X) \cap H(v,w)\cap\F^n = \conv_{\F}(X) \cap
H(v,w).$$
\end{pf}

\subsection{}

We now prove a general result relating weak $\F$-faces and polyhedra.

\begin{thm}\label{T1}
Suppose that $\conv_\R(X)$ is a polyhedron for $X \subset \F^n$. Then, $Y
\subset X$ is a weak $\F$-face if and only if $Y = F \cap X$ for some
face $F$ of $\conv_\R(X)$. Moreover, $\conv_\R(Y) = F$ in this case.
\end{thm}

\noindent In particular, faces of polyhedra are weak $\R$-faces, using $\F
= \R$.

\begin{pf}
First, suppose that $Y = F \cap X$ for some face $F$ of $\conv_{\R}(X)$.
By \lemref{L2}, one can find a linear functional $\varphi \in (\R^n)^*$
such that $\varphi(u) \geq \varphi(v)$ for all $u \in F$ and $v \in
\conv_{\R}(X)$. Let $x_0 \in U \subset X$, and suppose $u \in
\conv_{\F}(Y) \cap \relint(\conv_{\F}(U))$.

We can write $u = \sum_{y\in Y} s_y y$ with $s_y \in \F_+$ and
$\sum_{y\in Y} s_y = 1$, and, thus, $\varphi(u) = \varphi(F)$. By
\lemref{L1}, $ u = \sum_{j=0}^m r_jx_j$ for some $r_j \in \F_{>0}$ and
$x_j \in U$. Applying $\varphi$, we have
$$\varphi(F) = \varphi(u) = \sum_{j=0}^m r_j \varphi(x_j) \leq
\sum_{j=0}^m r_j \varphi(F) = \varphi(F).$$

\noindent Since $r_0$ is positive, $\varphi(x_0) = \varphi(F)$, so $x_0
\in F \cap X = Y$. Since $x_0 \in U$ was arbitrary, $U \subset Y$, so $Y
= F \cap X$ is a weak $\F$-face of $X$.

Now, let $Y$ be a weak $\F$-face of $X$, and let $F$ be the smallest face
of $\conv_{\R}(X)$ such that $Y \subset F$. If $\# F \cap X =1$, then $Y
= F\cap X$ and we are done.  Suppose that $\# F \cap X >1$. Since $F$ is
minimal and $\conv_{\R}(X)$ is a polyhedron, the interior of $F$ must
contain an element $y \in \conv_{\F}(Y)$. Let $x \in F \cap X$. If $x =
y$, then it is clear that $x \in Y$. 

Suppose that $x \neq y$. Then, by \lemref{L1}, $y = r_0x + \sum_{i=1}^m
r_ix_i$ for some $r_i \in \F_{>0}$ and $x_i \in F \cap X$. In particular,
$y \in \conv_{\F}(Y) \cap \relint(\conv_\F(F \cap X))$. Since $Y$ is a
weak $\F$-face of $X$, this gives that $F \cap X \subset Y$.

Finally, given $Y = F \cap X$ for some face $F$ of $\conv_\R(X)$, clearly
we have $\conv_\R(Y) \subset F$. Conversely, given $f \in F \subset
\conv_\R(X)$, $1 \cdot f = \sum_i a_i x_i$ for some $x_i \in X$, with
$0 \leq a_i$ adding up to 1. Now use \propref{P1} with $\F = \R$: since
$F$ is a weak $\R$-face of the polyhedron (by the remark following
the statement of this result), hence each $x_i \in F$. But then $x_i \in
F \cap X = Y$, so $f \in \conv_\R(Y)$ as desired.
\end{pf}

\subsection{}

We now study positive weak $\F$-faces. We start with an equivalent
characterization.

\begin{lem}\label{L4}
For all subsets $X \subset \R^n$ and subfields $\F\subset\R$, the
positive weak $\F$-faces of $X$ are the weak $\F$-faces $Y \subset X$
such that $Y$ is a weak $\F$-face of $X\cup\{0\}$ and $0 \notin
\conv_{\F}(Y)$.
\end{lem}

\begin{proof}
First, suppose that $Y$ is a positive weak $\F$-face of $X$. It follows
easily from the definition that $Y$ is a weak $\F$-face of $X\cup\{0\}$.
Suppose that $0 \in \conv_{\F}(Y)$, and let $U = Y$. Then, it is clear
that $\conv_{\F}(Y) \cap \relint(\conv_{\F}(U\cup\{0\})) =
\relint(\conv_{\F}(Y)) \neq \emptyset$, which contradicts $Y$ being a
positive weak $\F$-face of $X$.

Now, suppose $Y$ is a weak $\F$-face for both $X$ and $X\cup\{0\}$ such
that $0 \notin \conv_{\F}(Y)$. Let $U \subset X$. If $\conv_{\F}(Y) \cap
\relint(\conv_{\F}(U \cup \{ 0 \})) \neq \emptyset$, then $U \cup\{0\}
\subset Y$ since $Y$ is a weak $\F$-face of $X \cup\{0\}$. However, this
is impossible since $0 \not\in \conv_{\F}(Y)$. Thus, $Y$ is a positive
weak $\F$-face of $X$.
\end{proof}

To connect these results to the results in \cite{CKR}, we prove the
following proposition.

\begin{prop}\label{P1}
Let $X \subset \R^n$ and $\F$ a subfield of $\R$.
\begin{enumerit}
\item A subset $Y$ is a weak $\F$-face of $X$ if and only if
$$\sum_{y \in Y} m_y y = \sum_{x \in X} r_x x,\ m_y, r_x \in \F_+\
\forall y \in Y, x \in X \mbox{ and } \sum_{y \in Y} m_y = \sum_{x \in X}
r_x \implies x \in Y \mbox{ if } r_x \neq 0.$$

\item A subset $Y$ is a positive weak $\F$-face of $X$ if and only if (i)
holds and
$$\sum_{y \in Y} m_y y = \sum_{x \in X} r_x x \implies \sum_{y \in Y} m_y
\leq \sum_{x \in X} r_x.$$
\end{enumerit}
\end{prop}

\noindent \begin{pf} 
\begin{enumerit}
\item ($\Leftarrow$) Suppose $U \subset X$ and $u \in \conv_\F(Y)\cap
\relint(\conv_{\F}(U))$. Let $x_0 \in U$. By \lemref{L1} and the
definition of $\conv_\F(Y)$, we can write $u = \sum_{y \in Y} m_y y =
r_0x_0 + \sum_{j = 1}^m r_jx_j$ for some $x_j\in U$, where $m_y \in
\F\cap[0,1]$ for all $y \in Y$, $r_j \in \F\cap(0,1)$ for $j =
0,\dots,m$, and $\sum_{y\in Y} m_y = \sum_{j=0}^m r_j = 1$. Then, $r_0
\neq 0$, so $x_0 \in Y$. Since $x_0$ was arbitrary, $U \subset Y$.

\noindent ($\Rightarrow$)  Suppose that $u = \sum_{y\in Y} m_y y =
\sum_{x\in X} r_x x$ and $\sum_{y\in Y} m_y  = \sum_{x\in X} r_x>0$ with
$m_y, r_x \in \F_+$ for all $y \in Y$ and $x \in X$. 

Let $U = \{x \in X \ | \ r_x \neq 0\}$, and consider $u' =
\dfrac{1}{\sum_{x\in X} r_x} u \in \conv_{\F}(U)$. It is clear that $u'
\in \conv_{\F}(Y)$, so it suffices to show that $u' \in
\relint(\conv_{\F}(U))$. Furthermore, since each $x \in \conv_{\F}(U)$ is
a convex sum of a finite number of elements in $U$, it suffices to check
that for every $x_0 \in U$, there exists $y_0 \in \conv_{\F}(U)$ and $r_0
\in \F\cap(0,1)$ such that $u' = r_0x_0 + (1 - r_0)y_0$. By construction,
we have $u' = \sum_{x \in U} r_x' x$ with $r_x' \in \F\cap(0,1)$ and
$\sum_{x \in U} r_x' = 1$. Letting $r_0 = r_{x_0}'$, we have
$$u' = r_0 x_0 + \sum_{x \neq x_0} r_x' x = r_0 x_0 + (1 - r_0) \sum_{x
\neq x_0} \frac{r_x'}{1 - r_0} x.$$

\noindent It is easy to check that $y = \sum_{x\neq x_0}
\frac{r_x'}{1-r_0} x \in \conv_{\F}(U)$. 
In particular, $\conv_{\F}(Y) \cap \relint(\conv_{\F}(U)) \neq
\emptyset$, so $U = \{x \in X \ | \ r_x \neq 0\} \subset Y$. 

\item ($\Leftarrow$) Suppose that $\conv_{\F}(Y)\cap \relint
(\conv_{\F}(U \cup \{ 0 \})) \neq \emptyset$ for some $U \subset X$. Let
$u \in \conv_{\F}(Y) \cap \relint(\conv_{\F}(U \cup \{ 0 \}))$. By
\lemref{L1}, there exist $r_0, r_1, \dots, r_n \in \F \cap(0,1)$ and
$x_1, \dots, x_n \in U$ such that $\sum_{j=0}^n r_j = 1$, and
$$u = r_0\cdot 0 + \sum_{j=1}^n r_j x_j = \sum_{j=1}^n r_j x_j.$$

\noindent Similarly, there exist $m_y \in \F\cap[0,1]$ such that $\sum_{y
\in Y} m_y =1$, and $u = \sum_{y \in Y} m_y y$. However, this gives $u =
\sum_{y \in Y} m_y y = \sum_{j=1}^n r_j x_j$ with $\sum_{j=1}^n r_j = 1 -
r_0 < 1 = \sum_{y \in Y} m_y$, which is impossible. Thus, $\conv_{\F}(Y)
\cap \relint(\conv_{\F}(U \cup \{ 0 \})) = \emptyset$ for all $U \subset
X$.

\noindent ($\Rightarrow$) Let $u = \sum_{y \in Y} m_y y = \sum_{x\in X}
r_x x$ with $m_y, r_x \in \F\cap[0,\infty)$. Let $U = \{x \in X \ | \ r_x
\neq 0\}\neq \emptyset$, and suppose that $\sum_{x\in X} r_x < \sum_{y\in
Y}m_y$. Define $r_0 = \sum_{y \in Y} m_y - \sum_{x\in X}r_x >0$. Then, $u
= r_0\cdot 0 + \sum_{x\in U} r_x x$. 

Let $u' = \dfrac{1}{\sum_{y \in Y} m_y} u$. Since $r_0 + \sum_{x\in U}
r_x = \sum_{y\in Y} m_y$, $u' \in \conv_{\F}(U\cup\{0\})$. In fact, by an
argument similar to that in part $(i)$, $u' \in
\relint(\conv_{\F}(U \cup \{0\}))$. However, this gives $u' \in
\conv_{\F}(Y)\cap \relint(\conv_{\F}(U \cup\{0\}))$, which contradicts
$Y$ being a positive weak $\F$-face of X. Therefore, $\sum_{y\in Y}m_y
\leq \sum_{x\in X} r_x$.
\end{enumerit}
\end{pf}

Finally, this equivalent formulation of the positive weak $\F$-faces
allows us to explain the terminology.

\begin{thm}\label{T44}
Suppose $\conv_\R(X \cup \{ 0 \})$ is a polyhedron for $X \subset \F^n$.
Then $Y \subset X$ is a positive weak $\F$-face of $X$ if and only if $Y$
maximizes in $X$ some linear functional $\varphi\in(\R^n)^*$ which has a
positive maximum on $X$.

In particular, if $0$ is in the interior of $\conv_\R(X)$, then a subset
$Y$ is a positive weak $\F$-face of $X$ if and only if $Y \neq X$ and $Y$
is a weak $\F$-face of $X$.
\end{thm}

\begin{pf}
If $Y$ is a positive weak $\F$-face of $X$, then $Y$ is a positive weak
$\F$-face of $X \cup \{ 0 \}$, and hence also a weak $\F$-face of $X \cup
\{ 0 \}$ (both statements follow from the definitions), which does not
contain $0$ by \lemref{L4}. Hence by \thmref{T1}, $Y = F \cap (X \cup \{ 0 \})
= F \cap X$, for some face $F$ of $\conv_\R(X \cup \{ 0 \})$. Suppose $F$
maximizes the linear functional $\varphi$ in the polyhedron. Now if $0 \in F$,
then $0 \in F \cap (X \cup \{ 0 \}) = Y$, which contradicts \lemref{L4}. Thus,
$Y = F \cap X$ maximizes $\varphi$ in $X \cup \{ 0 \}$, and $0 \notin F$.
Hence $\varphi(Y) > \varphi(0) = 0$.

Conversely, choose $\varphi \in (\R^n)^*$ which is maximized in $X$
precisely on $Y$ and $\varphi(Y) >0$. (In particular, $Y$ is a weak
$\F$-face by \thmref{T1} and \lemref{L2}.) Suppose that $\sum_{y \in Y}
m_y y = \sum_{x \in X}r_x x$. Applying $\varphi$, we have
$$\varphi(Y) \sum_{y \in Y} m_y = \sum_{y \in Y} m_y \varphi(y) = \sum_{x
\in X} r_x \varphi(x) \leq \varphi(Y) \sum_{x \in X} r_x.$$

\noindent Since $\varphi(Y) >0$, this gives $\sum_{y \in Y} m_y \leq
\sum_{x\in X}r_x$, and $Y$ is a positive weak $\F$-face of $X$ by
\propref{P1}.

Finally, suppose that $0 \in \relint_\R(\conv_\R(X))$. The result is
clear if $X = \{ 0 \}$, so now suppose otherwise. Since $0$ is an
interior point, $0 \in \conv_\R(X \setminus \{ 0 \}) \cap \F^n$, so by
\propref{P2}, $0 \in \conv_\F(X \setminus \{ 0 \}) \subset
\conv_\F(X)$. Then, $X$ is not a positive weak $\F$-face of itself,
by \lemref{L4}. Since every positive weak $\F$-face is a weak
$\F$-face, it now suffices to prove that every proper weak
$\F$-face of $X$ is a positive weak $\F$-face.

Let $Y \subsetneq X$ be a (proper) weak $\F$-face of $X$. By Theorem
\ref{T1}, $Y = F \cap X$, for some proper face $F$ of $\conv_\R(X) =
\conv_\R(X \cup \{ 0 \})$. Since $0$ is an interior point, $0 \notin F$, so
$0 \notin Y = F \cap (X \cup \{ 0 \})$. By \lemref{L2}, $Y \subset X$
maximizes some linear functional $\varphi$ on $X \cup \{ 0 \}$, and $0
\notin Y$. Hence $\varphi(Y) > \varphi(0) = 0$, and we are done by the
first part of this result.
\end{pf}

\section{Application to representation theory}\label{S5}

We can now show one of our main results, using the above theory.

\begin{proof}[Proof of \thmref{T2}]
Suppose $\wt V \neq \{ 0 \}$. By the Decomposition Theorem \ref{T25}, the
sets $\conv_\R(\wt V)$ and $\conv_\R(\{ 0 \} \cup \wt V)$ are polytopes.
The result follows from \thmref{T1} and \thmref{T44}, once we show that
the origin is in the $\F$-relative interior of $\conv_\F(\wt V)$, for all $\F$.

First note that the vector $\rho_V := \sum_{\mu \in \wt V} \mu$ is
$W$-invariant, since $\wt V$ is stable under $W$. Then $s_i(\rho_V) =
\rho_V$, so $(\rho_V, \alpha_i) = 0\ \forall i$. Thus, $\rho_V = 0$. Now
given $y = \sum_{\mu \in \wt V} r_\mu \mu \in \conv_\F(\wt V)$, define
\[ z = \frac{1}{|\wt V| - 1} \sum_{\mu \in \wt V} (1-r_\mu) \mu \in
\conv_\F(\wt V). \]

\noindent Then $\rho_V = 0 = ty + (1-t)z$, where $t = \frac{1}{|\wt V|}
\in \F \cap (0,1)$. Hence $0 = \rho_V \in \relint(\conv_\F(\wt V))$.
\end{proof}

\subsection{}

We now prove the following result, before using it to show \thmref{T32}.
We introduce the following notation: given $\lambda \in \liehr^*$, define
$I_\lambda$ to be the union of those graph components of the Dynkin
diagram of $\lie g$, which are not disjoint from $\supp(\lambda)$.

\begin{prop}\label{P51}
Fix $0 \neq \lambda \in P^+$ and $J \subset I$.
Then $\wtvla{J} \subsetneq \wtvla{}$ if and only if $I_\lambda \nsubseteq
J$, if and only if $\max_{\wtvla{}} \rhovla{J} > 0$.
(Hereafter, we abuse notation, whereby $\mu \in \liehr^*$ denotes the
functional $(\mu,-)$.)
\end{prop}

\begin{proof}
We first make the following claim:
\begin{equation}\label{E51}
\wtvla{J} = \wtvla{J \cap I_\lambda}.
\end{equation}

Let us show the claim first. Clearly $\wtvla{J \cap I_\lambda} \subset
\wtvla{J}$. Next, suppose $\mu = \lambda - \sum_i a_i \alpha_i$ is any
weight of $V(\lambda)$. Then there is some $f \in U(\lie{n}^-)_{\mu -
\lambda}$ such that $f v_\lambda$ is a nonzero weight vector.
Since $U(\lie{n}^-)$ is the subalgebra of $U(\lie{g})$ generated by the
$x_{\alpha_i}^-$ ($i \in I$), write $f$ as a $\Bbb{C}$-linear combination
of monomial words (each of weight $\mu - \lambda$). Then at least one
such monomial word $x_{\alpha_{i_k}}^- \dots x_{\alpha_{i_2}}^-
x_{\alpha_{i_1}}^-$ does not kill any highest weight vector $0 \neq
v_\lambda \in V(\lambda)_\lambda$.

The claim is proved if we show that $a_i = 0\ \forall i \notin
I_\lambda$. Suppose not. Then there exists $1 \leq j \leq k$ such that
$i_j \notin I_\lambda$. Choose the minimal such $j$. Also note that
$x_{\alpha_{i_j}}^- \dots x_{\alpha_{i_1}}^- v_\lambda \neq 0$. Now since
$x_{\alpha_{i_j}}^-$ commutes with $x_{\alpha_{i_l}}^-$ for all $0<l<j$
(by the defining relations), we get:
\[ x_{\alpha_{i_{j-1}}}^- \dots x_{\alpha_{i_1}}^- (x_{\alpha_{i_j}}^-
v_\lambda) \neq 0, \]

\noindent whence $x_{\alpha_{i_j}}^- v_\lambda \neq 0$. Then, this is
a nonzero weight vector in the simple module $V(\lambda)$ of weight
$\lambda - \alpha_{i_j} \neq \lambda$, so this vector cannot be maximal
either; i.e., it is not killed by all of $\lie n^+$. Now $\lie n^+$ is
generated by $\{ x_{\alpha_i}^+ : i \in I \}$. For $i \neq i_j$,
$x_{\alpha_i}^+$ commutes with $x_{\alpha_{i_j}}^-$, so
\[ x_{\alpha_i}^+ (x_{\alpha_{i_j}}^- v_\lambda) = x_{\alpha_{i_j}}^-
\cdot x_{\alpha_i}^+ v_\lambda = 0. \]

\noindent Hence we must have: $x_{\alpha_{i_j}}^+ \cdot
x_{\alpha_{i_j}}^- v_\lambda \neq 0$. The left-hand side equals
$\lambda(h_{\alpha_{i_j}}) v_\lambda$ by standard computations, so
$\lambda(h_{\alpha_{i_j}}) \neq 0$. However, this is a contradiction
since $\lambda(h_{\alpha_i}) = 0$ for all $i \not\in I_{\lambda}$. Thus
the claim is proved, and $a_i = 0\ \forall i \not\in
I_{\lambda}$.\medskip

We are now ready to prove the result.
We first show two of the cyclic implications (more precisely, we show
their contrapositives). If $J \supset I_\lambda$, then by \eqref{E51},
\[ \wtvla{J} = \wtvla{J \cap I_\lambda} = \wtvla{I_\lambda} = \wtvla{I} =
\wtvla{}, \]

\noindent and we are done. Next, $\rhovla{I} = \rho_V = 0$ (see the proof
of \thmref{T2}), so we have: $\max_{\wtvla{}} \rhovla{I}$ $= \max 0 = 0$.

Finally, suppose $I_\lambda \nsubseteq J$; we prove that $\max_{\wtvla{}}
\rhovla{J} > 0$.
Since each weight is in $\lambda - Q^+$, $\rhovla{J} = |\wtvla{J}|
\lambda - \sum_{j \in J_1} m_j \alpha_j$ for some positive integers $m_j$
and some subset $J_1 \subset J$.
Since $I_\lambda \nsubseteq J$, there exists a graph component $I_j
\subset I_\lambda$ in the Dynkin diagram for $\lie{g}$, such that $I_j
\nsubseteq J$. We first show the following\smallskip

\noindent {\bf Claim.} There exists $j_0 \in I_j \subset I_\lambda$, such
that $(\rhovla{J}, \alpha_{j_0}) > 0$.

\begin{proof}
We have two cases. First, suppose that $I_j \cap J_1 = \emptyset$. Since
$\supp(\lambda) \cap I_j \neq \emptyset$, choose $j_0 \in I_j$ such that
$(\lambda, \alpha_{j_0}) > 0$. Now since $J_1 \cap I_j = \emptyset$, we
also have $(\alpha_i, \alpha_{j_0}) = 0\ \forall i \in J_1$. Then
$(\rhovla{J}, \alpha_{j_0}) = |\wtvla{J}| (\lambda, \alpha_{j_0})
> 0$.

On the other hand, if $I_j \cap J_1 \neq \emptyset$, then since $I_j$ is
connected, choose $j_0 \in I_j \setminus J_1$, that is adjacent to at
least one element $i_0 \in J_1$. Now
\[ (\rhovla{J}, \alpha_{j_0}) = |\wtvla{J}| (\lambda, \alpha_{j_0}) -
\sum_{j \in J} m_j (\alpha_j, \alpha_{j_0}) \geq 0 + \sum_{j \in J_1
\setminus \{ i_0 \} } m_j \cdot 0 - m_{i_0} (\alpha_{i_0}, \alpha_{j_0}),
\]

\noindent and this is strictly positive because $i_0, j_0$ are connected
by an edge in $I_j$.
\end{proof}

Returning to the proof of the result, since $\lambda = \sum_{i \in
\supp(\lambda)} (\lambda,\alpha_i) \omega_i$, $\lambda = \sum_{i \in
I_\lambda} a_i \alpha_i$ with all $a_i \in \Bbb{Q}_{>0}$ by
\cite[Exercise 13.8]{H}. We now compute:
\[ \max_{\wtvla{}} \rhovla{J} \geq (\rhovla{J}, \lambda) = \sum_{i \in
I_\lambda} a_i (\rhovla{J}, \alpha_i) \geq a_{j_0} (\rhovla{J},
\alpha_{j_0}) > 0. \]
\end{proof}

We now show another of our main results. We need some more notation.

\begin{defn}
Define, for any $J \subset I$,
\[ \Delta_J := \{ \alpha_j : j \in J \}, \quad \Delta := \Delta_I, \quad
\Omega_J := \{ \omega_j : j \in J \}, \quad \Omega := \Omega_I. \]

\noindent Now given $X \subset \liehr^*$, define $X(\lambda)$ to be:
\[ X(\lambda) := \{ x \in X : (x, \lambda) \geq (x', \lambda)\ \forall x'
\in X \} \subset X. \]
\end{defn}

\begin{rem}
It is not hard to show that $X(\lambda)$ is a weak $\F$-face of $X$ for
all $\lambda$ and all $\F$, and that if $\lambda(x) > 0$ for some $x \in
X$, then $X(\lambda)$ is a positive weak $\F$-face.
\end{rem}

\begin{proof}[Proof of \thmref{T32}]
By \thmref{T2} and \lemref{L2}, (iii) $\implies$ (ii). By \thmref{Tvin}
for $J = J_\lambda = I$, (ii) $\implies$ (i). It remains to show that (i)
$\implies$ (iii) (and the second part of the theorem). Since $w$ acts
linearly on $\liehr^*$ and $(,)$ is $W$-invariant, it suffices to prove
that (i) $\implies$ (iii) for $w=1$.

We now show that $\wtvla{J} = (\wtvla{})(\rhovla{J})$. First,
$\wtvla{J}$ is $W_J$-stable, hence so is $\rhovla{J}$. But then
$(\rhovla{J}, \alpha_i) = 0\ \forall i \in J$, whence
$(\rhovla{J},-)$ is constant on $\wtvla{J}$. Next, that the maximum
value is positive for proper subsets $\wtvla{J}$ was shown in
\propref{P51}.

Now let us suppose, as in the proof of \propref{P51}, that
$\rhovla{J} = |\wtvla{J}| \lambda - \sum_{j \in J_1} m_j \alpha_j$
for positive $m_j \in \Bbb{Z}$ and some $J_1 \subset J$. Thus, $\wtvla{J}
= \wtvla{J_1}$. Now if $i \notin J, j \in J_1$, then $(\alpha_j,
\alpha_i) \leq 0$ (since $J_1 \subset J$), so $(\rhovla{J}, \alpha_i)
\geq 0$ since $\lambda \in P^+$. In particular, $\rhovla{J} \in P^+$ from
above. In turn, this implies that $\lambda \in (\wtvla{})(\rhovla{J})$,
and from the previous paragraph, we conclude: $\wtvla{J} \subset
(\wtvla{})(\rhovla{J})$.

Now suppose $\nu \in \wtvla{}$ maximizes $\rhovla{J}$. We need to
show that $\nu \in \wtvla{J}$. We write $\nu = \lambda - \sum_{i \in I}
r_i \alpha_i$ for $r_i \in \zz$, and compute:
\[ (\rhovla{J}, \lambda) = (\rhovla{J}, \nu) =
(\rhovla{J}, \lambda) - \sum_{i \notin J_1} r_i (\rhovla{J},
\alpha_i) \leq (\rhovla{J}, \lambda), \]

\noindent since $(\rhovla{J}, \Delta_{J_1}) = 0$ from above. Now
define $J_2 := \{ i \notin J_1 : r_i > 0 \}$. The preceding equation
implies that $(\rhovla{J}, \alpha_i) = 0\ \forall i \in J_2$, so,
\[ |\wtvla{J}| (\lambda, \alpha_i) - \sum_{j \in J_1} m_j (\alpha_j,
\alpha_i) = 0. \]

\noindent Since $m_j > 0\ \forall j \in J_1$, $(\lambda,\alpha_i) =
(\alpha_j, \alpha_i) = 0\ \forall i \in J_2, j \in J_1$.

Now let $w_2$ be the longest element of the subgroup $W_{J_2}$ of $W$.
Consider $w_2(\nu) \in \wtvla{}$, where $\nu = \lambda - \sum_{j\in
J_1} r_j \alpha_j - \sum_{i \in J_2} r_i\alpha_i$. By the previous
paragraph, $w_2(\lambda) = \lambda$ and $w_2(\alpha_j) = \alpha_j\
\forall j \in J_1$ - and by its definition, $w_2(\alpha_i) \in
-\Delta_{J_2}\ \forall i \in J_2$.
Since $r_i \neq 0$ for $i \in J_2$, this gives $w_2(\nu) \notin \lambda -
\zz \Delta$ unless $J_2 = \emptyset$. Thus, $J_2 = \emptyset$, and $\nu
\in \wtvla{J_1} = \wtvla{J}$. Hence $(\wtvla{})(\rhovla{J}) =
\wtvla{J}$, and the theorem is proved.
\end{proof}

\begin{rem}
At this point, we note that \thmref{T32} does not hold for general
finite-dimensional $\lie{g}$-modules. For example, let $\lie{g}$ be of
type $A_2$, and consider the module $V = V(2\omega_2) \oplus V(\omega_1 +
\omega_2)$. It is easy to see that $\{\omega_1+\omega_2\}$ is a weak
$\F$-face of $\wt V$ for all $\F$. However, the subset of $\wt V$ that
maximizes the linear functional $(\omega_1+\omega_2, -)$ is the subset
$\{\omega_1 + \omega_2, 2\omega_2\}$.
\end{rem}

We now show a small result that helps classify {\it all} maximizer
subsets inside $\wtvla{}$, for $0 \neq \lambda \in P^+$. Given any
$\varphi \in \liehr^*$, the nondegeneracy of the Killing form implies
that $\varphi = (\nu,-)$, and there exists $w_\nu \in W$ such that
$w_\nu(\nu)$ is in the dominant Weyl chamber, i.e., in $\Bbb{R}_+
\Omega$.

\begin{lem}\label{L3}
Fix $0 \neq \lambda \in P^+$. Then for all $\nu \in \liehr^*$,
\[ (\wtvla{})(\nu) = w_\nu^{-1}(\wtvla{I \setminus \supp(w_\nu(\nu))}),
\]

\noindent and this map from $\liehr^*$ to the weak $\F$-faces of
$\wtvla{}$ is surjective:
\[ w(\wtvla{J}) = (\wtvla{})(w(\nu))\ \forall \nu \in \Bbb{R}_{>0}
\Omega_{I \setminus J}. \]
\end{lem}

\noindent In particular, $w(\wtvla{J}) = (\wtvla{})(w(\rho_{I \setminus
J}))\ \forall w,J$. Moreover, Theorem \ref{T33} helps determine the answer
to the question: For which (dominant) $\mu,\nu$ are the maximizer sets
the same?

\begin{proof}
First observe that since $(,)$ is $W$-invariant and $\wtvla{}$ is
$W$-stable,
\[ w(\wtvla{}(\nu)) = (\wtvla{})(w(\nu))\ \forall w \in W, \nu \in
\liehr^*. \]

Thus, it is enough to show the first claim for dominant $\nu$ (and $w_\nu
= 1$). Now, if $\nu = \sum_i a_i \omega_i$ with $a_i \geq 0\ \forall
i$ and $\mu = \lambda - \sum_{j \in I} (2b_j/(\alpha_j, \alpha_j))
\alpha_j$ with $b_j \geq 0\ \forall j$, then
\[ (\nu,\mu) = (\nu, \lambda) - \sum_{i,j \in I} a_i b_j \frac{2
(\omega_i,\alpha_j)}{(\alpha_j, \alpha_j)} = (\nu,\lambda) - \sum_{i \in
I} a_i b_i \leq (\nu,\lambda), \]

\noindent with equality if and only if $a_i b_i = 0\ \forall i$. This
precisely means that given $\nu$, we must have $b_i = 0\ \forall i \in
\supp(\nu)$, whence we arrive at $\wtvla{I \setminus \supp(\nu)}$.
Conversely, given that $\wtvla{J}$ is the maximizer (once again ignoring
the $w \in W$), we should have $a_i = 0\ \forall i \in J$, whence
$\supp(\nu) = I \setminus J$.
\end{proof}

\subsection{}

It remains to show the last result. Once again, we need some
preliminaries before proving it. Recall the definition of $\rho_Y$ from
\thmref{T32}.

\begin{prop}\label{P53}
Suppose $0 \neq \lambda \in P^+$.
\begin{enumerate}
\item Suppose $Y$ is a $W_J$-stable subset of $\wtvla{J}$ for some fixed
$J \subset I$. Then $|Y| \rhovla{J} = |\wtvla{J}| \rho_Y$.

\item The only $W_J$-invariant vector inside the face
$\conv_\R(\wtvla{J})$ is $\frac{1}{|\wtvla{J}|} \rhovla{J}$, which
is the center of the face.
\end{enumerate}
\end{prop}

\begin{proof}
The second part follows from the first, since if $x \in
\conv_\R(\wtvla{J})$ is $W_J$-invariant, then $x = \sum_i a_i y_i$ for
$y_i \in \wtvla{J}$ and $a_i \in (0,1)$ (and $\sum_i a_i = 1$). However,
\[ x = \frac{1}{|W_J|} \sum_{w \in W_J,\ i} a_i w(y_i), \]

\noindent whence $x$ is an $\Bbb{R}_+$-linear combination of $\rho_{Y_j}$
for distinct $W_J$-orbits $Y_j \subset \wtvla{J}$. Let us write this as:
$x = \sum_j b_j (\frac{1}{|Y_j|} \rho_{Y_j})$, with $\sum_j b_j = 1$
(because the coefficients above added up to 1). Using this and the first
part, we then get
\[ x = \sum_j b_j \frac{1}{|\wtvla{J}|} \rhovla{J} =
\frac{1}{|\wtvla{J}|} \rhovla{J}. \]

It remains to show the first part.
First, if $Y \subset \wtvla{J}$ is $W_J$-stable (and nonempty), then
$\rho_Y$ is fixed by $W_J$ since every $w \in W_J$ permutes $Y$.
Now, write $\rho_Y = |Y| \lambda - \sum_{j \in J} a_j \alpha_j$,
for some $a_j \in \zz$. Then, since $\rho_Y$ is $W_J$-invariant, we
get: $(\rho_Y, \alpha_j) = 0\ \forall j \in J$, which gives us a system
of $|J|$ linear equations in the $|J|$ variables $\{ a_j / |Y| \}$ -
namely,
\[ \sum_{j \in J} (a_j / |Y|) (\alpha_j, \alpha_i) = (\lambda, \alpha_i)\
\forall i \in J. \]

\noindent We now claim that the coefficients of the $a_j / |Y|$ are
precisely the entries of the ``symmetrized" Cartan matrix for $\lie{g}$,
in the rows and columns corresponding to $J \subset I$. But all principal
minors of a symmetrized Cartan matrix of finite type are positive, so
this matrix is nonsingular, which gives a unique (rational) solution to
the above system. The uniqueness implies that if we start with
$\rhovla{J} = |\wtvla{J}| \lambda - \sum_{i \in J} a'_i \alpha_i$, we
would get: $a'_i / |\wtvla{J}| = a_i / |Y|\ \forall i \in J$.
Thus, $\lambda - (1/|\wtvla{J}|) \rhovla{J} = \lambda - (1/|Y|)
\rho_Y$, and we are done. (Clearing the denominator of $|Y|$ also enables
us to include the case when $Y$ is the empty set, and $\rho_Y = 0$.)
\end{proof}

We conclude this paper with the proof of our last main result.

\begin{proof}[Proof of Theorem \ref{T33}]
If $\wtvla{I_1} = \wtvla{I_2}$, then the half-sums of all the elements are
clearly equal too: $\rhovla{I_1} = \rhovla{I_2}$. Conversely,
if $\rhovla{I_1} = \rhovla{I_2}$, then by \thmref{T32},
\[ \wtvla{I_1} = (\wtvla{})(\rhovla{I_1}) =
(\wtvla{})(\rhovla{I_2}) = \wtvla{I_2}. \]

Next, if $W_{I_1}(\lambda) = W_{I_2}(\lambda)$, then, since
$W_{I_i}(\lambda) \subset \wtvla{I_i}$ are $W_{I_i}$-stable (for
$i=1,2$), applying \propref{P53} twice gives
\[ \frac{1}{|\wtvla{I_1}|} \rhovla{I_1} =
\frac{1}{|W_{I_1}(\lambda)|} \sum_{x \in W_{I_1}(\lambda)} x =
\frac{1}{|W_{I_2}(\lambda)|} \sum_{x \in W_{I_2}(\lambda)} x =
\frac{1}{|\wtvla{I_2}|} \rhovla{I_2}. \]

Hence, $\rhovla{I_2} \in \Bbb{Q}_{>0} \rhovla{I_1}$, and
their maximizer subsets in $\wtvla{}$ coincide. By \thmref{T32},
$\wtvla{I_1} = \wtvla{I_2}$.

It remains to show the converse. Suppose that $\wtvla{I_1} =
\wtvla{I_2}$. Recall that these sets of weights are precisely the weights
of the modules $\bu(\lie g_{I_1})v_{\lambda}$ and $\bu(\lie
g_{I_2})v_{\lambda}$, respectively, where $0 \neq v_{\lambda}$ is a
highest weight vector of $V(\lambda)$.

Consider $\conv_{\R}(\wtvla{I_j})$ as the weight polytope of $\bu(\lie
g_{I_j})v_{\lambda}$ for $j = 1,2$. Since $\lie g_{I_1}$ and $\lie
g_{I_2}$ are both semisimple, we can apply \thmref{Tvin} to these
polytopes. In particular, we see that the set of vertices of
$\conv_{\R}(\wtvla{I_j})$ is precisely $W_{I_j}(\lambda)$. Since
$\wtvla{I_1} = \wtvla{I_2}$, these polytopes are equal, so they must have
the same vertices; i.e., $W_{I_1}(\lambda) = W_{I_2}(\lambda)$.
\end{proof}

\end{document}